\documentclass[a4paper,10pt, twoside]{scrartcl} 
\DeclareOldFontCommand{\rm}{\normalfont\rmfamily}{\mathrm}
\DeclareOldFontCommand{\sf}{\normalfont\sffamily}{\mathsf}
\DeclareOldFontCommand{\tt}{\normalfont\ttfamily}{\mathtt}
\DeclareOldFontCommand{\bf}{\normalfont\bfseries}{\mathbf}
\DeclareOldFontCommand{\it}{\normalfont\itshape}{\mathit}
\DeclareOldFontCommand{\sl}{\normalfont\slshape}{\@nomath\sl}
\DeclareOldFontCommand{\sc}{\normalfont\scshape}{\@nomath\sc}
\usepackage[english]{babel}
\usepackage[utf8]{inputenc}
\usepackage{cite}
\usepackage{leftidx,tensor}
\usepackage{amsmath}
\usepackage{amsthm}
\usepackage{mathtools}\usepackage{amsfonts}
\usepackage{amssymb}
\usepackage{fancyhdr}
\usepackage{tabularx}
\usepackage{geometry}
\usepackage{enumitem}
\usepackage{layout}
\usepackage{subfig}
\usepackage{floatflt}
\usepackage[T1]{fontenc}
\usepackage{bbm}
\usepackage[usenames,dvipsnames]{color}
\usepackage{titlesec}
\usepackage[subfigure,titles]{tocloft}
\usepackage[pdfpagelabels=true, colorlinks, linkcolor = blue , citecolor = blue, filecolor = blue, urlcolor = blue]{hyperref}

\makeatletter
\def\l@lstlisting#1#2{\@dottedtocline{1}{0em}{1em}{\hspace{1,5em} Lst. #1}{#2}}
\makeatother
\DeclareMathAlphabet\mathcalbm{OMS}{cmsy}{b}{n}

\newtheoremstyle{normal}
{10pt}
{10pt}
{}
{}
{\bfseries}
{}
{0em}
{\bfseries{\thmname{#1}\thmnumber{ #2}\thmnote{\hspace{0em}(#3)\newline}}}

\newtheoremstyle{standard}  
  {10pt}   
  {}   
  {\itshape}  
  {}       
  {\bfseries} 
  {:}         
  {0.2cm}  
  {\bfseries{\thmname{#1}\thmnumber{ #2}\thmnote{ \hspace{0em}(#3)}}}          

\newtheoremstyle{mittitel}  
  {10pt}   
  {}   
  {\itshape}  
  {}       
  {\bfseries} 
  {:}         
  {0.2cm}  
  {\bfseries{\thmname{#1}\thmnumber{ #2}\thmnote{ \hspace{0em}(#3)\newline}}}          

\DeclareMathOperator{\ft}{\mathcal{F}}
\DeclareMathOperator{\et}{\mathcal{E}}
\DeclareMathOperator{\s}{\mathcal{S}}
\DeclareMathOperator{\D}{\mathcal{D}}

\DeclareMathOperator{\ii}{\textit{\textbf{i}}}
\DeclareMathOperator{\jj}{\textit{\textbf{j}}}

\newcommand{\supp}{\mathop{\mathrm{supp}}}

\begin{document}
\frenchspacing

\thispagestyle{empty}

\begin{center}
\Large{\textbf{SPECTRAL ASYMPTOTICS FOR KREIN-FELLER-OPERATORS WITH RESPECT TO RANDOM RECURSIVE CANTOR MEASURES}}
\\[15pt]
\normalsize{LENON A. MINORICS\footnote{ Institute of Stochastics and Applications, University of Stuttgart, Pfaffenwaldring 57, 70569 Stuttgart, Germany, Email: Lenon.Minorics@mathematik.uni-stuttgart.de}}
\end{center}

\titleformat{\section}{\large\filcenter\scshape}{}{1em}{\thesection. \hspace{0em}}
\titleformat{\subsection}[runin]{\bfseries}{}{0pt}{\thesubsection. \hspace{0em}}
\titleformat{name=\section,numberless}[block]{\large\scshape\centering}{}{0pt}{}

\theoremstyle{standard}
 
\newtheorem{theorem}{Theorem}[section]
\newtheorem{proposition}[theorem]{Proposition} 
\newtheorem{lemma}[theorem]{Lemma}
\newtheorem{corollary}[theorem]{Corollary} 
\newtheorem{definition}[theorem]{Definition} 
\newtheorem{construction}[theorem]{Construction}
\newtheorem{remark}[theorem]{Remark}
\newtheorem{example}[theorem]{Example}
\newtheorem{condition}[theorem]{Condition}
\newtheorem{assumption}[theorem]{Assumption}
\newtheorem{acknowledgement}[theorem]{Acknowledgement}

\vspace*{8pt}
\textbf{Abstract.} We study the limit behavior of the Dirichlet and Neumann eigenvalue counting function of generalized second order differential operators $\frac{d}{d \mu} \frac{d}{d x}$, where $\mu$ is a finite atomless Borel measure on some compact interval $[a,b]$. We firstly recall the results of the spectral asymptotics for these operators received so far. Afterwards, we give the spectral asymptotics for so called random recursive Cantor measures. Finally, we compare the results for random recursive and random homogeneous Cantor measures.

\section{Introduction}
It is well known that $f \in C^0([a,b], \mathbb{R})$ possesses a $L_2$ weak derivative $g \in \mathcal{L}_2(\lambda^1,[a,b])$, where $\lambda^1$ denotes the one dimensional Lebesgue measure, if and only if 
\begin{align*}
f(x) = f(a) + \int_a^x g(y) \, d y.
\end{align*}
Replacing the one dimensional Lebesgue measure by some measure $\mu$ leads to a generalized $L_2$ weak derivative depending on the measure $\mu$. Therefore, we let $\mu$ be a finite non-atomic Borel measure on some interval $[a,b]$, $-\infty < a < b < \infty$. 
The \textit{$\mu$-derivative} of $f: [a,b] \longrightarrow \mathbb{R}$ for which $f^\mu \in \mathcal{L}_2(\mu)$ exists such that
\begin{align*}
f(x) = f(a) + \int_a^x f^\mu(y) \, d\mu(y) ~~~ \text{ for all } x \in [a,b]
\end{align*}
is defined as the unique equivalence class of $f^\mu$ in $L_2(\mu)$. We denote this equivalence class by $\frac{d f}{ d \mu}$. The \textit{Krein-Feller-operator} $\frac{d}{d \mu} \frac{d}{d x} f$ is than given as the $\mu$-derivative of the $\lambda^1 _{|_{[a,b]}}$-\\derivative of $f$. 

This operator were introduced for example in \cite{Ito65}. \cite{Kue80}, \cite{Kue86}, \cite{Loe91}, \cite{Loe93} investigate on properties of the generated stochastic process, called quasi or gap diffusion, and related objects.
\\

As in e.g. \cite{Arz14}, \cite{Fuj87}, we are interested in the spectral asymptotics for generalized second order differential operators $\frac{d}{d \mu} \frac{d}{d x}$ with Dirichlet or Neumann boundary conditions, i.e. we study the equation
\begin{align}
\frac{d}{d \mu} \frac{d}{d x} f = - \lambda f \label{eigenequ}
\end{align}
with
\begin{align*}
f(a) = f(b) = 0 ~~~~~ \text{or} ~~~~~ f'(a) = f'(b) = 0. 
\end{align*}

For a physical motivation, we consider a flexible string which is clamped between two points $a$ and $b$. If we deflect the string, a tension force drives the string back towards its state of equilibrium. Mathematically, the deviation of the string is described by some solution $u$ of the one dimensional wave equation
\begin{align*}
\frac{\rho(x)}{F} \frac{\partial^2 u(t,x)}{\partial t^2} = \frac{\partial^2 u(t,x)}{\partial x^2}, ~~~ x \in [a,b], ~ t \in [0,\infty)
\end{align*}
with Dirichlet boundary condition $u(t,a) = u(t,b) = 0$ for all $t$. Hereby, $\rho$ is given as the density of the mass distribution of the string and $F$ as the tangential acting tension force. To solve this equation, we make the ansatz $u(t,x) = \psi(t) \, \phi(x)$ and receive 
\begin{align*}
\frac{\psi '' (t)}{F \, \psi(t)} = \frac{\phi''(x)}{\phi (x) \, \rho(x)} = - \lambda,
\end{align*}
for some constant $\lambda \in \mathbb{R}$. In the following, we only consider the equation

\begin{align*}
\frac{\phi''(x)}{\phi  (x) \rho(x)} = - \lambda.
\end{align*}
Thus, we have
\begin{align*}
\phi'(t) - \phi '(a) = - \lambda \int_a^t \phi(y) \, d\mu (y),
\end{align*}
where $\mu$ is the mass distribution of the string. In other words,
\begin{align}
\frac{d}{d \mu} \frac{d}{d x} \phi = - \lambda \, \phi. \label{eigequ}
\end{align}
This equation no longer involves the density $\rho$, meaning that we can reformulate the problem for singular measures $\mu$.  Such a solution $\phi$ can be regarded as the shape of the string at some fixed time $t$. Up to a multiplicative constant, the natural frequencies of the string are given as the square root of the eigenvalues of \eqref{eigequ}.

In Freiberg \cite{Fre03} analytic properties of this operator are developed. There, it is shown that $-\frac{d}{d \mu} \frac{d}{d x}$ with Dirichlet or Neumann boundary conditions has a pure point spectrum and no finite accumulation points. Moreover, the eigenvalues are non-negative and have finite multiplicity. \\ We denote the sequence of Dirichlet eigenvalues of $ -\frac{d}{d \mu} \frac{d}{d x}$ by $\left(\lambda^\mu_{D,n}\right)_{n \in \mathbb{N}}$ and the sequence of Neumann eigenvalues by $\left(\lambda^\mu_{N,n}\right)_{n \in \mathbb{N}_0}$, where we assort the eigenvalues ascending and count them according to multiplicities. Let 
\begin{align*}
N_D^\mu (x) \coloneqq \# \left \{i \in \mathbb{N}: ~ \lambda^\mu_{D,i} \le x \right\} ~~~ \text{ and } ~~~ N_N^\mu(x) \coloneqq \# \left \{i \in \mathbb{N}_0: ~ \lambda^\mu_{N,i} \le x \right\}.
\end{align*}
$N_D^\mu$ and $N_N^\mu$ are called the Dirichlet and Neumann eigenvalue counting function of $-\frac{d}{d \mu} \frac{d}{d x}$, respectively. 
The problem of determining $\gamma > 0$ such that
\begin{align}
N_{D/N}^\mu(x) \asymp x^\gamma, ~~~ x \rightarrow \infty,
\end{align}
is an extension of the analogous problem for the one dimensional Laplacian. The following theorem is a well-known result of Weyl \cite{Wey15}.
\begin{theorem} \label{weyls law}
Let $\Omega \subseteq \mathbb{R}^n$ be a domain with smooth boundary $\partial \Omega$. Consider the eigenvalue problem
\begin{align*}
\begin{cases}
- \Delta_{n,\Omega} u &= \lambda u ~ \text{ on } ~\Omega, \\
u_{|\partial\Omega} &= 0,
\end{cases}
\end{align*}
where $\Delta_{n,\Omega}$ denotes the Laplace operator on $\Omega$. Then, for the Dirichlet eigenvalue counting function $N_D^{({n,\Omega})}$ of $\Delta_{n,\Omega}$ it holds that
\begin{align}
N_D^{({n,\Omega})}(x) = (2\pi)^{-n} \, c_n \, \textit{vol}_n(\Omega) \, x^{n/2} + o\left(x^{n/2}\right), ~~~ x\rightarrow \infty, \label{weyls formula}
\end{align}
hereby $c_n$ denotes the volume of the $n$-dimensional unit ball.
\end{theorem}
Choosing $\mu = \lambda_{|_{[a,b]}}^1$ leads to
\begin{align*}
N^{\mu}_{D}(x) = N_D^{(1,(a,b))}(x) \asymp x^{1/2}, ~~~ x \rightarrow \infty,
\end{align*}
which gives the leading order term in the Weyl asymptotics as in Theorem \ref{weyls law}. 
\eqref{weyls formula} motivates the definition of the spectral dimension 
\begin{align}
\frac{d_s(\Omega)}{2} \coloneqq \lim_{\lambda \rightarrow \infty} \frac{\log N_D^{(n,\Omega)}(\lambda)}{\log \lambda}. \label{spec dim}
\end{align}
Which leads to 
\begin{align*}
d_s(\Omega) = n
\end{align*}
in Theorem \ref{weyls law}.
Many authors before studied the expression \eqref{spec dim} for generalized Laplacians on p.c.f. fractals, e.g. \cite{Fre15}, \cite{Ham00}, \cite{Kig93}.
In this paper, we investigate on this expression for the Krein-Feller-operator on so called random recursive Cantor sets. Therefore, we call the limit
\begin{align*}
\gamma \coloneqq \gamma(\mu) \coloneqq \lim_{\lambda \rightarrow \infty} \frac{\log N_D^\mu(\lambda)}{\log \lambda}
\end{align*}
the \textit{spectral exponent} of the corresponding Krein-Feller-operator. \\

The spectral asymptotics for Krein-Feller-operators with respect to self similar measures was developed by Fujita \cite{Fuj87}, more general by Freiberg \cite{Fre05} and with respect to random (and deterministic) homogeneous Cantor measures by Arzt \cite{Arz14}. 

We give an example of a random recursive Cantor set and a corresponding random recursive Cantor measure. In Section \ref{constr frac} we define the general class. The fractal is constructed as follows: we subdivide the unit interval with probability $p$ into three intervals with equal lengths, where we remove the open middle third interval and with probability $1-p$ into five intervals with equal lengths, where we remove the open second and fourth interval. In the next step, we subdivide the remaining intervals independent from each other likewise and continue the procedure. The fractal under consideration is the limiting set, called random $\frac{1}{3}$-$\frac{1}{5}$-recursive Cantor set.
\vspace*{1em}
\begin{figure}[ht]
\centering
\includegraphics[scale=0.7]{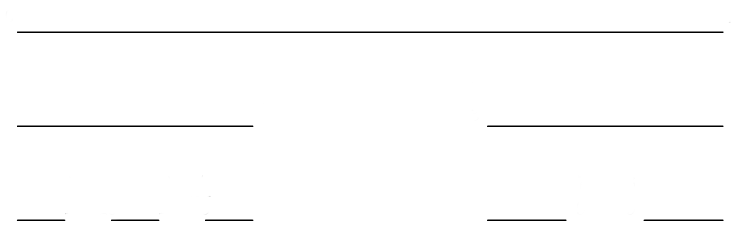}
\caption{First two approximation steps of one possible random $\frac{1}{3}$-$\frac{1}{5}$-recursive Cantor set}
\end{figure}

Afterwards, we construct probability measures $\mu_n$, $n \in \mathbb{N}$ such that $\mu_n$ is a weighted Lebesgue measure those support is given by the $n$-th approximation step of the random $\frac{1}{3}$-$\frac{1}{5}$-recursive Cantor set. To this end, let $m^{(j)} = \left(m_1^{(j)},..,m_{N_{j}}^{(j)}\right)$, $j=1,2$, $N_1 = 2, N_2 =3$ be vectors of weights, i.e. $\sum_{i=1}^{N_j} m_i^{(j)} = 1$, $m_i \in (0,1)$, $i=1,..,N_j$, $j=1,2$.
$\mu_1$ weights the left remaining interval by $m_1^{(1)}$ and the right by $m_2^{(1)}$, if we subdivided the unit interval into three parts, else it weights the left interval by $m_1^{(2)}$, the middle interval by $m_2^{(2)}$ and the right by $m_3^{(2)}$. $\mu_2$ weights an interval by the weight of the predecessor interval multiplied by the weight according to the procedure for $n=1$. Recursively, we continue this construction.\\
A random recursive Cantor measure $\mu^{\left(\frac{1}{3},\frac{1}{5}\right)}$ corresponding to the $\frac{1}{3}$-$\frac{1}{5}$-recursive Cantor set is given as the weak limit of the sequence $\left( \mu_n \right)_{n \in \mathbb{N}}$.
\begin{figure}[ht]
\centering
\includegraphics[scale=0.7]{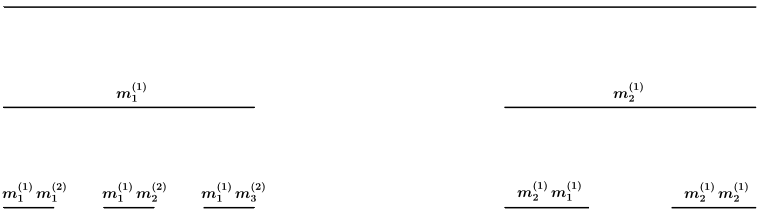}
\caption{First two approximation steps of $\mu^{\left(\frac{1}{3},\frac{1}{5}\right)}$}
\end{figure}

It turns out that under some regularity conditions for the solution $\gamma > 0$ of
\begin{align*}
\mathbb{E}\left( \sum_{i=1}^{N_{U_\emptyset}} \left(r_i^{(U_\emptyset)}\, m_i^{(U_\emptyset)}\right)^\gamma \right) = 1, 
\end{align*}
 there exists a constant $C>0$ and a random variable $W>0$ a.s., $\mathbb{E}W = 1$ such that
\begin{align}
N_{D/N}^\mu(x) \, x^{-\gamma} \longrightarrow C\,W ~~~~a.s. \label{nonlattice}
\end{align}
or there exists a deterministic periodic function $G$ such that
\begin{align}
N_{D/N}^{\mu}(x) = (G(\log(x)) + o(1))\,x^{\gamma}\,W ~~~~ a.s., \label{lattice}
\end{align}
where $\mu$ is a random recursive Cantor measure. Hereby $U_\emptyset$ is the unique ancestor of the underlying random tree, $N_{U_\emptyset}$ is the corresponding number of self similarities, $r_i^{(U_\emptyset)}$ are the corresponding scale factors and $m_i^{(U_\emptyset)}$ are the entries of the corresponding vector of weights.

Since the eigenvalue counting functions are branching processes, they fulfill a random version of the renewal equation of \cite{Fel66}. The constant $C$ in \eqref{nonlattice} is given as the limit of $\mathbb{E}\left(N_{D/N}^\mu(x) \, x^{-\gamma} \right)$. The random variable $W$ is the limit of the fundamental martingale of the underlying random population. The strict positivity of $W$ follows by an $x \log x$ argument, standard in branching theory.\\ It is an open question whether there exists a non-trivial example in \eqref{lattice} or not. 

For the random $\frac{1}{3}$-$\frac{1}{5}$-recursive Cantor set we thus receive that either \eqref{nonlattice} or \eqref{lattice} is satisfied,
where $\gamma > 0$ is the unique solution of
\begin{align*}
p \, \left( \left( \frac{m_1^{(1)}}{3} \right)^\gamma + \left( \frac{m_2^{(1)}}{3} \right)^\gamma \right) + (1-p) \, \left( \left( \frac{m_1^{(2)}}{5} \right)^\gamma + \left( \frac{m_2^{(2)}}{5} \right)^\gamma + \left( \frac{m_3^{(2)}}{5} \right)^\gamma \right) = 1.
\end{align*}
We denote by $\gamma_0$ the spectral exponent for $p=0$ and by $\gamma_1$ for $p=1$. For every $\gamma \in [\gamma_0,\gamma_1]$ there thus exists a $p$ such that $\gamma$ is the corresponding spectral exponent. Therefore, we can construct a tailored string those spectral exponent is an arbitrary $\gamma \in [\gamma_0,\gamma_1]$, where the support of this string is then given by some random $\frac{1}{3}$-$\frac{1}{5}$-recursive Cantor set.

The paper is organized as follows. In Section \ref{prelim} we give the definition of the operator under consideration and recap the important results received so far. Section \ref{branching} is dedicated to the C-M-J branching processes. The convergence results for these types of branching processes we need are given there. We use them to establish the spectral asymptotics for the eigenvalue counting functions. Then, in Section \ref{section asympotic}, we firstly define the measures under consideration and proof afterwards the main theorem. Finally, we compare the spectral exponent for random homogeneous and random recursive Cantor measures. It will be shown that iff $W\neq 1$ a.s., the spectral exponent for random recursive Cantor measures is strictly bigger than the spectral exponent for random homogeneous Cantor measures. We illustrate this fact by some examples.

\section{Preliminaries} \label{prelim}
\subsection{Definition of the Krein-Feller-Operator.} \label{ablmu}
Let $\mu$ be a finite non-atomic Borel measure on $[a,b]$, $-\infty < a < b < \infty$ and 
\begin{align*}
\D_1^\mu \coloneqq \bigg\{f: [a,b] \longrightarrow \mathbb{R}: ~ \exists ~ f^\mu \in ~&\mathcal{L}_2(\mu): \\& f(x) = f(a) + \int_a^x f^{\mu}(y) ~ d \mu (y), ~~~ x \in [a,b] \bigg\}.
\end{align*}
The $ \mu$\textit{-derivative} of $f$ is defined as the equivalence class of $f^\mu$ in $L_2(\mu)$. It is known (see \cite[Corollary 6.4]{Fre03}) that this equivalence class is unique. Thus, the operator 
\begin{align*}
\frac{d}{d \mu} : \D_1^\mu &\longrightarrow L_2(\mu), \\
f &\mapsto ~~ [f^\mu]_{\sim_\mu} 
\end{align*}
is well-defined.
Let
\begin{align*}
\D \coloneqq \D_2^{\mu , \lambda^1}\coloneqq \bigg\{ f \in C^1((a,b)) \cap C^0([a,b]) :& ~ \exists ~ (f')^\mu \in \mathcal{L}_2\left( \mu \right): \\& f'(x) = f'(0) + \int_a^x (f')^\mu(y) ~  d \mu (y), ~~~ x \in [a,b]  \bigg\}.
\end{align*}
The \textit{Krein-Feller-operator w.r.t.} $ \mu$ is given as 
\begin{align*}
\frac{d}{d \mu} \frac{d}{d x}: \D &\longrightarrow  L_2(\mu) \\
f &\mapsto ~~ [(f')^\mu]_{\sim_\mu}.
\end{align*}

\subsection{Spectral Asymptotics for Self-Similar and Random Homogeneous Cantor Measures.} \label{self sim and hom case}
As mentioned in the introduction, the spectral asymptotics for Krein-Feller-operators were discovered by \cite{Fuj87} and \cite{Arz14} for special types of measures. In this section we summarize some main results.  Firstly, we consider self-similar measures, treated in \cite{Fuj87}. Therefore, let $\s = \{S_1,...,S_N\}$, $N \ge 2$ be an iterated function system given by
\begin{align*}
S_i(x) = r_i \, x + c_i, ~~~ x \in [a,b],
\end{align*}
whereby $r_i \in (0,1)$, $c_i \in \mathbb{R}$ are constants such that the open set condition is satisfies, $S_i[a,b] \subseteq [a,b]$ for all $i$ and let $m = (m_1,...,m_N)$ be a vector of weights. As shown in \cite{Hut81}, there exists a unique non-empty compact set $C=C(\s) \subseteq [a,b]$ such that $\bigcup_{i=1}^N S_i(C) = C$ and a unique Borel probability measure $\mu = \mu(\s,m)$ such that $\mu = \sum_{i=1}^N m_i \, \mu \circ S_i^{-1}$. Moreover it holds $\supp \mu = C$. We call $C$ self-similar w.r.t. $\s$ and $\mu$ self-similar w.r.t. $\s$ and $m$. The Hausdorff dimension of $C$ is given by the unique solution $d \in [0,1]$ of $\sum_{i=1}^N r_i^{d} = 1$ and it holds $\mathcal{H}^{d}(C) \in (0,\infty)$. Moreover, if $m_i = r_i^{d}$ for all $i$, we have $\mu = \mathcal{H}^d(C)^{-1} \, \mathcal{H}^d_{|_{C}}$. In this setting, the spectral exponent of the corresponding Krein-Feller-operator is the unique solution $\gamma > 0$ of $\sum_{i=1}^N \left( m_i \, r_i \right)^\gamma = 1$. For references see \cite[Theorem 3.6]{Fuj87} and \cite[Theorem 4.1]{Fre05}. 
\\\\
In the following, we want to relax the self similarity of the set $C$ and the measure $\mu$. To this end, we take an index set $J$ and define to each $j \in J$ an IFS $\s^{(j)}=\left\{S_1^{(j)},...,S_{N_j}^{(j)} \right\}$. Then, we choose randomly $j_0 \in J$ (according to some probability distribution on $J$) and take the image of $[a,b]$ under $\s^{(j_0)}$. Next, we choose randomly $j_1 \in J$ (according to the same probability distribution) and take the image of $S_1^{(j_0)}[a,b],...,S_{N_{j_0}}^{(j_0)}[a,b]$ under $\s^{(j_1)}$. The limit of this construction is the fractal under consideration. More precise, let $J$ be a non-empty countable set. To each $j \in J$ let $\s^{(j)}=\left\{S_1^{(j)},...,S_{N_j}^{(j)} \right\}$, $N_j \in \mathbb{N}$ be such that 
\begin{align*}
S_i^{(j)}(x) = r_i^{(j)} \, x + c_i^{(j)}, ~~~ x \in [a,b], ~ i=1,...,N_j,
\end{align*}
where the constants $r_i^{(j)} \in (0,1)$, $c_i^{(j)} \in \mathbb{R}$ are chosen such that
\begin{align} \label{assumption IFS}
a = S_1^{(j)}(a) < S_1^{(j)} (b) \le S_2^{(j)}(a) < \cdots < S_{N_j}^{(j)}(b) = b. 
\end{align}
Further, we call $\xi = \left(\xi_1, \xi_2 ,... \right)$, $\xi_i \in J$ an environment sequence and define 
\begin{align*}
W_n \coloneqq \left\{1 ,..., N_{\xi_1} \right\} \times  \left\{1 ,..., N_{\xi_2} \right\} \times \cdots \times \left\{1 ,..., N_{\xi_n} \right\}, ~~~ n \in \mathbb{N}. 
\end{align*} 
The homogeneous Cantor set to a given environment sequence $\xi$ is
\begin{align*}
K^{(\xi)} \coloneqq \bigcap_{n=1}^\infty \bigcup_{w \in W_n} \left(S_{w_1}^{(\xi_1)} \circ S_{w_2}^{(\xi_2)} \circ \cdots \circ S_{w_n}^{(\xi_n)} \right)([a,b]). 
\end{align*}
Next, we define a measure $\mu^{(\xi)}$ on $[a,b]$ to a given environment sequence $\xi$, which generalizes the invariant measures, presented before. To this end, let $m^{(j)}= (m_1^{(j)},...,m_{N_j}^{(j)})$, $j \in J$ be a vector of weights. $\mu^{(\xi)}$ is defined as the week limit of the sequence of Borel probability measures $\left(\mu_n^{(\xi)}\right)_{n \in \mathbb{N}}$,
\begin{align*}
\mu_n^{(\xi)} \coloneqq \sum_{w \in W_n} m_{w_1}^{(\xi_1)} \cdots m_{w_n}^{(\xi_n)} \, \mu_0 \circ \left(S_{w_1}^{(\xi_1)} \circ \cdots \circ S_{w_n}^{(\xi_n)} \right)^{-1}, ~~~~~ \mu_0 \coloneqq \frac{1}{b-a} \, \lambda^1_{|_{[a,b]}}.
\end{align*}
$\mu^{(\xi)}$ is called homogeneous Cantor measure, corresponding to $K^{(\xi)}$. If $|J| = 1$, then the definition of invariant sets and measures coincide with $K^{(\xi)}$ and $\mu^{(\xi)}$.\\
\cite[Theorem 3.3.10]{Arz14} makes a statement about the spectral exponent of the Krein-Feller-operator with respect to $\mu^{(\xi)}$, where $\xi$ is a deterministic environment sequence. Here, we only consider the random case. Therefore, let $(\Omega,\mathcal{F},\mathbb{P})$ be a probability space and $\xi = (\xi_1,\xi_2,..)$ a sequence of i.i.d. $J$-valued random variables with $p_j \coloneqq \mathbb{P}(\xi_i = j)$. We denote the Dirichlet and Neumann eigenvalue counting function of the Krein-Feller-operator w.r.t. $\mu^{(\xi(\omega))}$ by $N_D^{(\xi(\omega))}$ and $N_N^{(\xi(\omega))}$, respectively. Further, if $|J| = \infty$, we need the following five technical assumptions:
\begin{enumerate}[label=\textbf{A.\arabic*}]
\item[] 
\begin{align}
\sup_{j \in J} N_j < \infty, \label{A1} \tag{A1}
\end{align} 
\vspace{-1cm}
\item[] 
\begin{align}
\inf_{j \in J} \min_{i=1,...,N_j} r_i^{(j)} m_i^{(j)} > 0, \label{A2} \tag{A2}
\end{align}
\vspace{-0.85cm}
\item[]
\begin{align}
\sup_{j \in J} \max_{i=1,...,N_j} r_i^{(j)} m_i^{(j)} < 1, \label{A3} \tag{A3}
\end{align}
\vspace{-1cm}
\item[] 
\begin{align} 
\prod_{j \in J, \atop \sum_{i=1}^{N_j} r_i^{(j)} m_i^{(j)} < 1} \sum_{i=1}^{N_j} r_i^{(j)} m_i^{(j)} > 0, \label{A4} \tag{A4}
\end{align}
\vspace{-1cm}
\item[]
\begin{align}
\prod_{j \in J, \atop \sum_{i=1}^{N_j} r_i^{(j)} m_i^{(j)} > 1} \sum_{i=1}^{N_j} r_i^{(j)} m_i^{(j)} < \infty. \label{A5} \tag{A5}
\end{align}
\end{enumerate}
Under these assumptions, we obtain:
\begin{theorem}[\cite{Arz14}, Corollary 3.5.1] \label{hom asym}
Let $\gamma_h > 0$ be the unique solution of
\begin{align*}
\prod_{j \in J} \left( \sum_{i=1}^{N_j} \left(r_i^{(j)} m_i^{(j)}\right)^{\gamma_h} \right)^{p_j} = 1.
\end{align*}
Then, there exist $C_1, C_2 > 0$, $x_0 > 0$ and $c_1(\omega), c_2(\omega) > 0$ such that
\begin{align*}
C_1 \, x^{\gamma_h} \, e^{-c_1(\omega)\sqrt{\log x \log \log x}} \le N_D^{\left(\xi(\omega) \right)}(x) \le  N^{\left(\xi(\omega)\right)}_N(x)  \le C_2 \, x^{\gamma_h} \, e^{ -c_2(\omega)\sqrt{\log x \log \log x}}
\end{align*}
for all $x > x_0$ almost surely.
\end{theorem}

\section{C-M-J Branching Processes} \label{branching} 
By the construction of random recursive Cantor sets, there is a natural relation to random labelled trees. We will be able to write the eigenvalue counting function as a sum over each node of the tree, counted by some random characteristic which leads to C-M-J branching processes. This method was also used in \cite{Ham00}. Nerman \cite{Ner81} used renewal theory, based on \cite{Fel66}, for some convergence results for C-M-J branching processes. These results can then be used to determine the asymptotic behaviour of the eigenvalue counting functions. 
\\\\
A C-M-J branching process is a stochastic process which counts individuals of a population according to some (maybe random) function $\phi$. We assume that the considered population has a unique ancestor, denoted by $\emptyset$. We say $\ii = (i_1,...,i_n)$ belongs to the $n$-th generation of the population, if the individual $\ii$ is the $i_n$-th child of the $i_{n-1}$-th child of the $...$ of the $i_1$-th child of the ancestor $\emptyset$. Since a mother can give birth to a child, we say $\boldsymbol{\tilde \iota}$ is the mother of $\ii$, if $ \boldsymbol{\tilde \iota} = (i_1,...,i_{n-1})$. 
The generation of $\ii$ is given by $|\ii|$.
Each individual has a reproduction rate, described by a random point process $\xi_{\ii}$ on $[0,\infty)$, i.e. an individual reproduces at time $t$ according to $\xi_{\ii} (t)$, for $t \in [0,\infty)$, whereby $\xi_{\ii} (t)$ denotes the $\xi_{\ii}$ measure of $[0,t]$. The birth time of $\ii$ is denoted by $\sigma_{\ii}$ and is given as 
\begin{align*}
\sigma_\emptyset &= 0, \\
\sigma_{\ii} &= \sigma_{\boldsymbol{\tilde \iota}} + \inf \left\{u \ge 0: \xi_{\boldsymbol{\tilde \iota}} (u) \ge i_n \right\}.
\end{align*} 
Every individual has a life time $L$. Therefore, it lives in the interval $[\sigma_{\ii},L + \sigma_{\ii})$ and dies at time $L + \sigma_{\ii}$. We define the tuple $(\xi, L, \phi)$ on some probability space $(\tilde{\Omega},\tilde{\mathcal{B}}, \tilde{\mathbb{P}})$. We call $(\xi_x,L_x,\phi_x)_x$ a \textit{general branching process}. Let
\begin{align*}
\mathcal{G}_n &\coloneqq \left\{(i_1,...,i_n): i_j \in \mathbb{N},~ j=1,...,n \right\}, \\[5pt]
\mathcal{G} &\coloneqq \{0\} \cup \left( \bigcup_{n=1}^\infty \mathcal{G}_n \right).
\end{align*}
The probability space on which we define the C-M-J branching processes is the product space
\begin{align}
(\Omega, \mathcal{B}, \mathbb{P}) = \prod_{\ii \in \mathcal{G}} (\Omega_{\ii}, \mathcal{B}_{\ii}, \mathbb{P}_{\ii}), \label{prob space}
\end{align}
where $(\Omega_{\ii}, \mathcal{B}_{\ii}, \mathbb{P}_{\ii})$ are copies of $(\tilde{\Omega},\tilde{\mathcal{B}}, \tilde{\mathbb{P}})$ and contain independent copies $(\xi_{\ii}, L_{\ii}, \phi_{\ii})$ of $(\xi, L, \phi)$. Thereby, we assume that $\phi: \Omega \times \mathbb{R}  \longrightarrow [0,\infty)$ is a product measurable, separable càdlàg function on $\mathbb{R}$.
The \textit{C-M-J branching process} to a given general branching process $(\xi_x,L_x,\phi_x)_x$ is defined by
\begin{align*}
Z^{\phi}_t \coloneqq \sum_{\ii \in \Sigma} \phi_{\ii}(t-\sigma_{\ii}),
\end{align*}
where $\Sigma$ is the trace of the underlying Galton-Watson process and $\phi_{\ii}(t) = 0$ for $t<0$. The interpretation of the process $Z^\phi$ depends on the random characteristic $\phi$. For $\phi \equiv 1$, $Z^\phi$ describes the total number of individuals born up to and including time $t$. In this case, we set $T_t \coloneqq Z^\phi_t$. Further, we define $\nu(t) \coloneqq \nu ([0,t]) \coloneqq \mathbb{E}(\xi(t))$ and we require that the following two properties hold:
\begin{enumerate}
\item
There exists an $\alpha >0 $ such that
\begin{align*}
\int_0^\infty e^{- \alpha t} \, d \nu (t) = 1.
\end{align*}
This parameter $\alpha$ is called \textit{Malthusian parameter} of the process.
\item
For the Malthusian parameter $\alpha$ holds
\begin{align*}
\int_0 ^\infty u \, e^{-\alpha u} \, d \nu (u) < \infty.
\end{align*}
\end{enumerate}
The following representation of $Z^\phi$ is useful for our consideration (see \cite{Jag75}):
\begin{align}
Z^\phi_t = \phi_\emptyset(t) + \sum_{i=1}^{\xi_\emptyset(t)} \leftidx{_{(i)}}Z_{t-\sigma_i}^\phi, ~~~ t \in [0, \infty),  \label{process sum presentation}
\end{align}
where $\left(\leftidx{_{(i)}} Z_t^\phi \right)_t $, $i=1,...,\xi_\emptyset(\infty)$ are i.i.d., distributed like $\left(Z_t^\phi \right)_t$. Also, $\left(\leftidx{_{(i)}} Z_t^\phi\right)_t $ is independent of $\xi_\emptyset$. If there will be no confusion, we will suppress the $\ii$ in $\phi_{\ii}$, $L_{\ii}$, etc. Further, we write $\xi ( \infty )$, if we mean $\xi ([0, \infty))$ and analogously for the other measures. The type of branching processes we consider is called supercritical, i.e. $\nu(\infty) > 1$. In this case the extinction probability is strictly less than 1 (see e.g. \cite[Theorem 2.3.1]{Jag75}). In our consideration each individual will have at least two offsprings and therefore the extinction probability is 0. By $\xi_\alpha$ we denote the Laplace-Stieltjes transformation with respect to $\alpha$ of $\xi$ and by $\nu_\alpha$ its expectation, i.e.
\begin{align*}
\xi_\alpha(t) &= \int_0^t e^{- \alpha s} \, d \xi (s), \\[3pt]
\nu_\alpha(t) &= \mathbb{E}(\xi_\alpha(t)).
\end{align*}
In the following we order the individuals according to their birth times, that is, if $\ii$ is the $n$-th individual of the population and
\begin{align*}
\sigma_{\ii} < \sigma_{(\ii,i)}, 
\end{align*}
for some $i \in \mathbb{N}$ and there exists no individual $\jj$ such that
\begin{align*}
\sigma_{\ii} < \sigma_{\jj} < \sigma_{(\ii,i)},
\end{align*}
then $(\ii,i)$ is the $(n+1)$-th individual. If we have several births at the same time, we sort them according to an arbitrary rule. We write $\ii_{(n)}$ for the $n$-th individual of the population.

For our main result, we need to introduce a random variable $W$ which is the almost sure limit of a martingale $(R_n)_{n \in \mathbb{N}}$. Therefore, we define a filtration $(\mathcal{A}_n)_{n \in \mathbb{N}}$ on the probability space $(\Omega,\mathcal{B},\mathbb{P})$ as follows:
For $\jj \in \mathcal{G}$ let $P_{\jj}$ be the projection of $(\Omega,\mathcal{B})$ onto $\left(\Omega_{\jj},\mathcal{B}_{\jj}\right)$. Then, $\mathcal{A}_n$ is defined as the smallest $\sigma$-algebra (on $\Omega$) such that
\begin{align*}
\left\{\omega \in \Omega : ~ \ii_{(1)}(\omega) = \jj_1 ,...,\ii_{(n)}(\omega) = \jj_n \right\} \in \mathcal{A}_n ~~~ \text{ for all } \jj_1,...,\jj_n \in \mathcal{G}
\end{align*}
and
\begin{align*}
A \cap \left\{\omega \in \Omega: ~ \jj \in \left\{\ii_{(1)}(\omega) ,..., \ii_{(n)}(\omega) \right\}  \right\} \in \mathcal{A}_n ~~~ \text{ for all } A \in P_{\jj}^{-1}(\mathcal{B}), ~ \text{ for all } \jj \in \mathcal{G}.
\end{align*}
We interpret $\mathcal{A}_n$ as the biography of the first $n$ individuals.
By construction $\sigma_{\ii_{(n)}}$ is $\mathcal{A}_{n-1}$ measurable. Further, we have that 
$\leftidx{_{\left(\ii_{(k)}\right)}} Z_t^\phi$ and $\xi_{\ii_{(k)}}$ are independent of $\mathcal{A}_n$ for all $k > n$, $t \in \mathbb{R}$. We remark that analogous results hold for $\mathcal{A}_{T_t}$ (for individuals born after time $t$ such that their parents are born before or at time $t$), where $T_t$ is a stopping time with respect to the constructed filtration for fixed $t$. 
Let $\mathcal{H}(n)$ be the set of the first $n$ individuals of the population and
\begin{align*}
R_0 &\coloneqq 1 \\
R_n &\coloneqq 1 + \sum_{\ii \in \mathcal{H}(n)} \sum_{i=1}^{\xi_{\ii}(\infty)} e^{- \alpha \sigma_{(\ii, i)}} - \sum_{\ii \in \mathcal{H}(n)} e^{- \alpha \sigma_{\ii}}, ~~~ n \in \mathbb{N}.
\end{align*}

\begin{theorem} \label{W bigger zero}
The process $(R_n)_{n \in \mathbb{N}}$ is a non-negative martingale with respect to \\ $(\mathcal{A}_n)_{n \in \mathbb{N}}$. Furthermore, there exists a random variable $W$ such that
\begin{align*}
R_n \stackrel{n \rightarrow \infty}{\longrightarrow} W ~~~ a.s.
\end{align*}
If 
\begin{align*}
\mathbb{E}\left(\xi_\alpha(\infty) \log^+ \xi_\alpha(\infty)\right) < \infty,
\end{align*}
then $W>0$ a.s., otherwise $W=0$ a.s.
\end{theorem}
\begin{proof}
\cite[Theorem 4.1]{Asm84}
\end{proof}

The case where $\phi_{\ii}$ depends on the whole line of descendants is discussed in \cite[Chapter 7]{Ner81}. There, it is shown that Theorem \ref{W bigger zero} also holds.

We need a strong law of large numbers for C-M-J branching processes. For reference see \cite{Cha14}. For this strong law, the branching process has to satisfy the following two conditions.

\begin{condition} \label{CB1}
There exists a non-increasing bounded positive integrable càdlàg function $g$ on $[0,\infty)$ such that
\begin{align*}
\mathbb{E} \left( \sup_{t \ge 0} \frac{\xi_\alpha(\infty)- \xi_\alpha(t)}{g(t)} \right) < \infty.
\end{align*}
\end{condition}

\begin{condition} \label{CB2}
There exists a non-increasing bounded positive integrable càdlàg function $h$ on $[0,\infty)$ such that
\begin{align*}
\mathbb{E}\left( \sup_{t \ge 0} \frac{e^{- \alpha t}\phi(t)}{h(t)} \right)< \infty .
\end{align*}
\end{condition}

\begin{theorem}[strong law of large numbers] \label{strong law}
Let $(\xi_x,L_x,\phi_x)_x$ be a general branching process with Malthusian parameter $\alpha$, where $\phi \ge 0 $ and $\phi(t) = 0$ for $t < 0$. Then,
\begin{enumerate}
\item If $\nu_\alpha$ is non-lattice,
\begin{align*}
e^{-\alpha t} Z^{\phi}(t) \rightarrow \nu_\alpha (\infty) \, W ~~~ a.s.
\end{align*}
\item
If $\nu_\alpha$ is lattice with span $T$, there exists a periodic function $G$ with period $T$ such that
\begin{align*}
Z^\phi(t) = (G(t) + o(1))\, e^{\alpha t} \, W ~~~ a.s.
\end{align*}
$G$ is given as
\begin{align*}
G(t) = T \cdot \frac{\sum_{j=-\infty}^{\infty} e^{-\alpha (t+jT)} \phi (t+jT)}{\int_0^\infty t e^{-\alpha t} \,d \nu (t)}. 
\end{align*}
\end{enumerate}
\end{theorem}

\section[5pt]{Spectral Asymptotics for General Recursive Cantor Measures} \label{section asympotic}
 
\subsection{Construction of General Recursive Cantor Measures.} \label{constr frac}

Let $J$ be a (possibly uncountable) index set. We define to each $j \in J$ an IFS $\s^{(j)}$. Therefore, let $N_j \in \mathbb{N}$, $ N_j \ge 2$. Then $\s^{(j)} = \left(S_1^{(j)},...,S_{N_j}^{(j)}\right)$, where we define $S_i^{(j)} : [a,b] \longrightarrow [a,b]$ by
\begin{align*}
S_i^{(j)}(x) \coloneqq r_i^{(j)} \, x + c_i^{(j)},
\end{align*}
for some $r_i^{(j)} \in (0,1)$, $c_i^{(j)} \in \mathbb{R}$, $i=1,...,N_j$ such that
\begin{align*}
a = S_1^{(j)}(a) < S_1^{(j)}(b) \le S_2^{(j)}(a) < S_2^{(j)}(b) \le \cdots \le S_{N_j}^{(j)}(a) <  S_{N_j}^{(j)}(b)=b.
\end{align*}
Furthermore, let $m^{(j)} = \left(m^{(j)}_1,...,m^{(j)}_{N_j}\right)$ be a vector of weights and thus, as in Chapter \ref{self sim and hom case}, an element of the index set $J$ identifies a tuple $\left(S^{(j)},m^{(j)}\right)$.\\
As in Chapter \ref{branching}, we construct a population I with unique ancestor, denoted by $\emptyset$. Every individual $\ii \in I$ identifies an element of $J$ which we also denote by $\ii$. The number of children of $\ii$ is $N_{\ii}$. For $\ii, \jj \in \mathcal{G}$, $\ii = (i_1,...,i_n)$, $\jj = (j_1,...,j_m)$ we define $\ii\jj \coloneqq (i_1,...,i_n,j_1,...,j_m)$ and, if $m>n$, $\jj|_n \coloneqq (j_1,...,j_n) $.
Let $I_n$ be the $n$-th generation of $I$.

For $\ii \in I_{n}$, $\ii = (i_1,...,i_n)$, we define
\begin{align*}
m_{\ii} &\coloneqq m_{i_1}^{(\emptyset)} \cdots m_{i_n}^{((i_1,...,i_{n-1}))}, \\
S_{\ii} &\coloneqq S_{i_1}^{(\emptyset)} \circ ... \circ S_{i_{n}}^{((i_1,...,i_{n-1}))}
\end{align*}
and we define analogously $S_{\ii}^{-1}$ as the composition of the preimages of the $S_i$.

For $n \in \mathbb{N}$ let
\begin{align*}
K_{n}^{(I)} \coloneqq \bigcup_{\ii \in I_n} S_{\ii}([a,b]).
\end{align*}
The limiting set $K^{(I)} \coloneqq \bigcap_{n=1}^\infty K_n^{(I)}$ is called \textit{recursive Cantor set}.

\begin{proposition}
The set $K^{(I)}$ is compact and contains at least countably infinitely many elements, namely $S_{(i_1,...,i_n)}(a)$ and $S_{(i_1,...,i_n)} (b)$, $i_1= 1,...,N_{\emptyset},...,i_{n} = 1,...,N_{(i_1,...,i_{n-1})}$, $n \in \mathbb{N}$.
\end{proposition}

\begin{proof}
Let $\ii = (i_1,...,i_n) \in I_{n}$. For $m \in \mathbb{N}$ let $\ii'$ and $\ii''$ be two individuals of the population such that $\ii' = \ii  \textbf{1}_m$, $\textbf{1}_m \coloneqq (1,...,1) \in \mathbb{R}^m $, $m \in \mathbb{N}$ and $i_1'',...,i_{n}'' = i_1,...,i_{n}$, $i''_k = N_{\left(i_1,...,i_{k-1},N_{(i_1,...,i_{k-1})}\right)}$ for $k = n+1,..,n+m$.  By definition, we have
\begin{align*}
S_{\ii'}(a) &= S_{\ii}(a), \\
S_{\ii''}(b) &= S_{\ii}(b).
\end{align*}
Thus, we have $S_{\ii}(a)$, $S_{\ii}(b) \in K_{n+m}^{(I)}$ for all $m \in \mathbb{N}$, which proofs the statement.
\end{proof}

By construction, we have
\begin{align}
K^{(I)} = \bigcup_{i=1}^{N_{\emptyset}} S_i^{(\emptyset)}\left(K^{(\theta_i I)}\right), \label{similarity}
\end{align}
where $\theta_i I$ denotes the subtree of $I$, rooted at $(i)$.

We define the recursive Cantor measures, analogously to the homogeneous Cantor measures. Let 

\begin{align*}
\mu_n^{(I)}(A) \coloneqq \sum_{\ii \in I_n} m_{\ii} \, \mu_0\left(S_{\ii}^{-1}(A)\right), ~~~ \mu_0(A) \coloneqq \frac{1}{b-a} \lambda^1_{|_{[a,b]}}(A)
\end{align*}

for all $A\in \mathfrak{B}([a,b])$.
The recursive cantor measure $\mu^{(I)}$ to given Cantor set coded by $I$ is defined as the weak limit of $\left( \mu^{(I)}_n \right)_{n \in \mathbb{N}}$. 
 
\begin{lemma}
For all $\ii \in I$ holds
\begin{align*}
\mu^{(I)} (S_{\ii}([a,b])) = m_{\ii}.
\end{align*}
\end{lemma}

\begin{proof}
We write $\mu = \mu^{(I)}$, $\mu_n = \mu_n^{(I)}$, $n \in \mathbb{N}$. Let $K_{\ii} \coloneqq S_{\ii}([a,b])$ for $\ii \in I$.Let $\ii \in I_{n}$, $\jj \in I_{n+m}$, $n,m \in \mathbb{N}$. 
Because of 
\begin{align*}
K_{\ii} \cap K_{\jj}= 
\begin{cases}
K_{\jj},& \text{ if } \jj|_n = \ii \\
\emptyset,& \text{ otherwise},
\end{cases}
\end{align*}
we get
\begin{align*}
&\mu_{n+m}(K_{\ii}) \\[5pt]
=&\sum_{\jj \in I_{n+m}}m_{\jj} \, \mu_0\left(S_{\jj}^{-1}(K_{\ii})\right) \\[5pt]
=&\sum_{\jj \in I_{n+m} \atop \jj|_n = \ii } m_{\jj} \, \mu_0\left(S_{\jj}^{-1}(K_{\ii})\right).
\end{align*} 

Because of
\begin{align*}
&\left(S_{(i_1,...,i_n,j_{n+1},...,j_{n+m})})^{-1}(K_{\ii}\right) \\[5pt] = 
&\left(S_{j_{n+1}}^{((i_1,...,i_n))} \circ \cdots \circ S_{j_{n+m}}^{((i_1,...,i_n,j_{n+1},...,j_{n+m-1}))}\right)^{-1} \circ S_{\ii}^{-1}   (K_{\ii}) \\[5pt] = 
&\left(S_{j_{n+1}}^{((i_1,...,i_n))} \circ \cdots \circ S_{j_{n+m}}^{((i_1,...,i_n,j_{n+1},...,j_{n+m-1}))}\right)^{-1} ([a,b]) \\[5pt] =& [a,b],
\end{align*}

we get 
\begin{align*}
\mu_{n+m}(K_{\ii}) &= \sum_{\jj \in I_{n+m} \atop \jj|_n = \ii } m_{\jj} = m_{\ii}.
\end{align*}

\end{proof}
Analogously to \eqref{similarity} holds 
\begin{align}
\mu^{(I)} = \sum_{i=1}^{N_{\emptyset}} m_i^{(\emptyset)} \, S_i^{(\emptyset)} \mu^{(\theta_i I)}, \label{measure similarity}
\end{align}
where $S_i^{(\emptyset)} \mu^{(I)}(A) \coloneqq \mu^{(I)}\left(\left(S_i^{(\emptyset)}\right)^{-1}(A)\right)$, $A \in \mathcal{B}([a,b])$. 
\begin{proof}
Let $A \in \mathcal{B}([a,b])$. Then, we get
\begin{align*}
&\sum_{i=1}^{N_{\emptyset}} m_i^{(\emptyset)} \, \mu_{n}^{(\theta_i I)}\left(\left(S_i^{(\emptyset)}\right)^{-1}(A)\right) \\   
=& \sum_{i=1}^{N_{\emptyset}} \sum_{i_1=1}^{N_{i}} \cdots \sum_{i_n=1}^{N_{(i,i_1,...,i_{n-1})}} m_{i}^{(\emptyset)} \, m_{i_1}^{(i)} \cdots m_{i_{n}}^{((i,i_1,...,i_{n-1}))} \, \mu_0\left(\left(S_{i,i_1,...,i_{n}}\right)^{-1}(A)\right) \\
=& \sum_{i_1=1}^{N_{\emptyset}} \cdots \sum_{i_{n+1}=1}^{N_{(i_1,...,i_{n})}} m_{i_1}^{(\emptyset)} \, m_{i_2}^{(i)} \cdots m_{i_{n+1}}^{((i,i_1,...,i_{n}))} \, \mu_0\left(\left(S_{i_1,...,i_{n+1}}\right)^{-1}(A)\right) \\
=& \mu_{n+1}^{(I)}(A).
\end{align*}
Taking the limit, we get the assertion.
\end{proof}
With \eqref{measure similarity} we get the following lemma.
\begin{lemma} \label{scale measure}
Let $i \in \{1,...,N_{\emptyset}\}$ and $A \in \mathcal{B}([a,b])$ with $A \subseteq S_i^{(\emptyset)}([a,b])$. Then, it holds
\begin{align*}
\mu^{(I)}(A) = m_i^{(\emptyset)} \, (S_i^{(\emptyset)} \mu^{(\theta_i I)})(A).
\end{align*}
\end{lemma}

\subsection{Scaling Properties.}
We establish a Dirichlet-Neumann-Bracketing with which we receive the characteristic $\phi$ for the C-M-J branching process under consideration. To this end, we need some scaling properties. 
\subsubsection{Scaling Property of the $\boldsymbol L_{\boldsymbol 2}$-Norm.}

\begin{lemma} \label{scaleprop E}
Let $f,g \in L_2\left(\mu^{(I)}\right)$. Then, 
\begin{align*}
\left\langle f,g \right\rangle_{L_2\left(\mu^{(I)}\right)} = \sum_{i=1}^{N_{\emptyset}} m_i^{(\emptyset)} \left\langle f \circ S_i^{(\emptyset)},g\circ S_i^{(\emptyset)} \right\rangle_{L_2\left(\mu^{\left(\theta_i I\right)}\right)}.
\end{align*}
\end{lemma}
\begin{proof}
We have $\supp \mu^{(I)} = K^{(I)}$. Together with Lemma \ref{scale measure}, we get 

\begin{align*}
\langle f,g \rangle_ {L_2(\mu^{(I)})} &= \int_{[a,b]} f \, g \, d \mu^{(I)} \\
&= \sum_{i=1}^{N_{\emptyset}} \int_{S_i^{(\emptyset)}([a,b])} f\, g \, d \mu^{(I)}  \\
&= \sum_{i=1}^{N_{\emptyset}} \int_{[a,b]} f \circ S_i^{(\emptyset)} \, g \circ S_i^{(\emptyset)} \, d\left(S_i^{(\emptyset)^{-1}} \mu^{(I)}\right) \\
&= \sum_{i=1}^{N_{\emptyset}} m_i^{(\emptyset)} \int_{[a,b]} f  \circ S_i^{(\emptyset)} \, g \circ S_i^{(\emptyset)} \, d \mu^{\left(\theta_i I\right)} \\
&= \sum_{i=1}^{N_{\emptyset}} m_i^{(\emptyset)} \langle f\circ S_i^{(\emptyset)},g\circ S_i^{(\emptyset)} \rangle_{L_2\left(\mu^{(\theta_i I)}\right)}.
\end{align*}
\end{proof}

\subsubsection{Scaling of the Eigenvalue Counting Function - Neumann Boundary Conditions.}

Let $(\mathcal{E}^{(I)},\mathcal{F})$ be  the Dirichletform on $L(\mu^{(I)})$, whose eigenvalues coincide with the Neumann eigenvalues of $-\frac{d}{d \mu^{(I)}} \frac{d}{d x}$. Namely, 
\begin{align*}
\mathcal{F} &= H^1(\lambda), \\
\mathcal{E}(f,g) &= \int_a^b f'(x) \, g'(x) \, dx,
\end{align*}
see \cite[Proposition 5.1]{Fre04/05}.
We write $N^{(I)}_N$ for the eigenvalue counting function of $(\mathcal{E},\mathcal{F})$, instead of $N_{(\mathcal{E},\mathcal{F})}$. To obtain the Neumann-Dirichlet-Bracketing, we define a new Dirichlet form $\left(\tilde{\mathcal{E}}^{(I)},\tilde{\mathcal{F}}^{(I)}\right)$, introduced in \cite[Chapter 3]{Arz14}. Let $\tilde{\mathcal{F}}^{(I)}$ be the set of all functions $f: [a,b] \longrightarrow \mathbb{R}$ with $f \circ S_i^{(\emptyset)} \in \mathcal{F}$ for all $i = 1,..., N_{\emptyset}$ and $f \big|_{\left(S_i^{(\emptyset)}(b), S_{i+1}^{(\emptyset)}(a)\right)} \in H^1\left(\lambda, \left(S_i^{(\emptyset)}(b), S_{i+1}^{(\emptyset)}(b)\right)\right)$ for all $i=1,...,N_{\emptyset}-1$. With \cite[Proposition 3.2.1]{Arz14} follows $\mathcal{F} \subseteq \tilde{\mathcal{F}}^{(I)}$, but $\tilde{\mathcal{F}}^{(I)} \nsubseteq \mathcal{F}$, because $f \in \tilde{\mathcal{F}}^{(I)}$ has not to be continuous on the boundary points of $S_i^{(\emptyset)}([a,b])$, $i=1,...,N_{\emptyset}$. For all $f,g \in \tilde{\mathcal{F}}^{(I)}$, we define 
\begin{align*}
\tilde{\mathcal{E}}^{(I)}(f,g) \coloneqq \sum_{i=1}^{N_{\emptyset}} \frac{1}{r_i^{(\emptyset)}} \mathcal{E}\left(f \circ S_i^{(\emptyset)}, g \circ S_i^{(\emptyset)}\right) + \sum_{i=1}^{N_{\emptyset}-1} \int_{S_i^{(\emptyset)}(b)}^{S_{i+1}^{(\emptyset)}(a)} f'(t) \, g'(t) \, dt.
\end{align*}

Due to \cite[Proposition 3.2.1]{Arz14} we then have for all $f,g \in \mathcal{F}$, $\tilde{\mathcal{E}}^{(I)}(f,g) = \mathcal{E}(f,g)$. Further, \cite[Proposition 2.2.2]{Arz14} implies that the embedding $\tilde{F}^{(I)} \hookrightarrow L_2(\mu^{(I)})$ is a compact operator and thus we can refer to the eigenvalue counting function of the Dirichletform $\left(\tilde{\mathcal{E}}^{(I)},\tilde{\mathcal{F}}^{(I)}\right)$. From now on we suppress the $I$ dependence of the Dirichletform $\left(\tilde{\mathcal{E}}^{(I)},\tilde{\mathcal{F}}^{(I)}\right)$.

\begin{proposition}
For all $x \ge 0$  holds 
\begin{align*}
N_{(\tilde{\mathcal{F}},\tilde{\mathcal{E}})}(x) = \sum_{i=1}^{N_{\emptyset}} N^{(\theta_i I)}_N\left(r_i^{(\emptyset)} m_i^{(\emptyset)}x\right).
\end{align*}
\end{proposition}
\begin{proof}
Let $f$ be an eigenfunction of $\left(\tilde{\mathcal{E}}, \tilde{\mathcal{F}},\mu^{(I)}\right)$ with eigenvalue $\lambda$, i.e. 
\begin{align*}
\et (f,g) = \lambda \, \langle f , g \rangle_{L_2\left(\mu^{(I)}\right)} ~~~ \text{ for all } g \in \tilde{\mathcal{F}}.
\end{align*}
Because $f,g \in L_2\left(\mu^{(I)}\right)$, we have with Lemma \ref{scaleprop E} 
\begin{align}
\begin{split}
&\sum_{i=1}^{N_{\emptyset}} \frac{1}{r_i^{(\emptyset)}} \mathcal{E}\left(f \circ S_i^{(\emptyset)}, g \circ S_i^{(\emptyset)}\right) + \sum_{i=1}^{N_{\emptyset}-1} \int_{S_i^{(\emptyset)}(b)}^{S_{i+1}^{(\emptyset)}(a)} f'(t) \, g'(t) \, dt \\  &= \lambda \, \sum_{i=1}^{N_{\emptyset}} m_i^{(\emptyset)} \, \left\langle f \circ S_i^{(\emptyset)},g \circ S_i^{(\emptyset)} \right\rangle_{L_2\left(\mu^{(\theta_i I)}\right)}. \label{efeq}
\end{split}
\end{align}
Now, we show that each summand on the left side equals each summand on the right side, respectively. Therefore, let $h \in \mathcal{F}$ and define for each $j \in \{1,...,N_{\emptyset}\}$
\begin{align*}
\tilde{h}_j(x) \coloneqq
\begin{cases}
h \circ S_j^{(\emptyset)^{-1}}(x), &\text{ if } x \in S_j^{(\emptyset)}([a,b]), \\
0, &\text{otherwise.}
\end{cases}
\end{align*}
Obviously, we have $\tilde{h}_j \in \ft, \tilde{h}_j \circ S_j^{(\emptyset)} = h$, for all $j \in \{1,...,N_{\emptyset}\}$ and $\tilde{h}_j \circ S_i^{(\emptyset)} = 0$ for $i \neq j$. Moreover, $\tilde{h}_j'\big|_{\left(S_i^{(\emptyset)}(b),S_{i+1}^{(\emptyset)}(a)\right)} = 0$, $j=1,...,N_{\emptyset}$, $i=1,..,N_\emptyset-1$. With $g = \tilde{h}_j$, we then have in \eqref{efeq}  
\begin{align*}
 \frac{1}{r_j^{(\emptyset)}} \, \mathcal{E}\left(f \circ S_j^{(\emptyset)}, h\right) = \lambda \, m_j^{(\emptyset)} \, \left\langle f \circ S_j^{(\emptyset)},h \right\rangle_{L_2\left(\mu^{\left(\theta_j I\right)}\right)}.
\end{align*}
Because this equation holds for all $h \in \mathcal{F}$, $f \circ S_j^{(\emptyset)}$ is an eigenfunction of the Dirichletform $\left(\mathcal{E}, \mathcal{F},\mu^{\left(\theta_jI\right)}\right)$ with eigenvalue  $r_j^{(\emptyset)}  m_j^{(\emptyset)}  \lambda$ for all $j = 1 ,..., N_{\emptyset}$.
\\\\
Now, let $\lambda > 0$, s.t. for $i=1,...,N_{\emptyset}$, $r_i^{(\emptyset)} m_i^{(\emptyset)} \lambda$ is an eigenvalue of $\left(\mathcal{E}, \mathcal{F},\mu^{(\theta_iI)}\right)$ with eigenfunction $f_i$, say. This means, 
\begin{align*}
\mathcal{E}(f_i,g) = r_i^{(\emptyset)} m_i^{(\emptyset)} \lambda \, \left\langle f_i, g \right\rangle_{L_2\left(\mu^{\left(\theta_i I\right)}\right)},
\end{align*}
for all $g \in \mathcal{F}$. Let
\begin{align*}
f(x) \coloneqq 
\begin{cases}
f_i \circ S_i^{(\emptyset)^{-1}}(x), &\text{if } x \in S_i^{(\emptyset)}([a,b]) \text{ for some } i \in \{1,...,N_{\emptyset}\} \\
0, &\text{otherwise}.
\end{cases}
\end{align*}
Then $f \in \ft$ and $f \circ S_i^{(\emptyset)} = f_i$, $i=1,...,N_{\emptyset}$ and therefore
\begin{align*}
\sum_{i=1}^{N_{\emptyset}} \frac{1}{r_i^{(\emptyset)}} \, \mathcal{E}\left(f \circ S_i^{(\emptyset)} ,g\right) =  \lambda \sum_{i=1}^{N_{\emptyset}} m_i^{(\emptyset)} \, \left\langle f \circ S_i^{(\emptyset)}, g \right\rangle_{L_2\left(\mu^{(\theta_i I)}\right)},
\end{align*}
for all $g \in \mathcal{F}$. Since for $\tilde{g} \in \ft$ we have by definition of $\ft$, $\tilde{g} \circ S_i^{(\emptyset)} \in\mathcal{F}$, $i=1,...,N_{\emptyset}$, we get
\begin{align*}
\sum_{i=1}^{N_{\emptyset}} \frac{1}{r_i^{(\emptyset)}} \, \mathcal{E}(f \circ S_i^{(\emptyset)} ,\tilde{g} \circ S_i^{(\emptyset)}) = \lambda \sum_{i=1}^{N_{\emptyset}} m_i^{(\emptyset)}  \, \left\langle f \circ S_i^{(\emptyset)}, \tilde{g} \circ S_i^{(\emptyset)} \right\rangle_{L_2\left(\mu^{(\theta_i I)}\right)}. 
\end{align*} 
But the left side of this equation is equal to $\et (f,\tilde{g})$, because $f'\big|_{\left(S_i^{(\emptyset)}(b),S_{i+1}^{(\emptyset)}(a))\right)} = 0$, for all $i=1,...,N_{\emptyset}-1$. With Lemma \ref{scaleprop E} we then have
\begin{align*}
\et (f,\tilde{g}) = \lambda  \, \langle f, \tilde{g} \rangle_{L_2\left(\mu^{(I)}\right)},
\end{align*}
for all $\tilde{g} \in \ft$. Therefore, $\lambda$ is an eigenvalue of $\left(\et , \ft , \mu^{(I)}\right)$ with corresponding eigenfunction $f$. Using this, we can easily conclude the claim. 
\end{proof}

\subsubsection{Scaling of the Eigenvalue Counting Function - Dirichlet Boundary Conditions.}

Let $(\mathcal{F}_0,\mathcal{E})$ be the Dirichlet form on $L_2\left(\mu^{(I)}\right)$ whose eigenvalues coincide with the Dirichlet eigenvalues of $-\frac{d}{d \mu^{(I)}} \frac{d}{d x}$. Meaning, $\mathcal{E}$ is defined as before and 
\begin{align*}
\mathcal{F}_0 \coloneqq \{f \in \mathcal{F}: ~ f(a)=f(b)=0\}.
\end{align*}
We write $N_D$ instead of $N_{(\mathcal{F}_0,\mathcal{E})}$. Again, we define a new Dirichletform $\left(\mathcal{E},\tilde{\mathcal{F}_0}^{(I)}\right)$ on $L_2\left(\mu^{(I)}\right)$ and suppress the $I$ dependence of 
\begin{align*}
\tilde{\mathcal{F}}^{(I)}_0 \coloneqq \left\{f \in F_0 : ~ f(x) = 0 \text{ for } x \in \left(S_i^{(\emptyset)}(b), S_{i+1}^{(\emptyset)}(a)\right),~ i=1,...,N_{\emptyset}-1\right\}.
\end{align*}
Further, we use the notation $\mathcal{E}$ for $\mathcal{E}\big|_{\ft _0 \times \ft _0}$.

\begin{proposition}
For all $x \ge 0$ we have
\begin{align*}
N_{\left(\mathcal{E},\ft_0,\mu^{(I)}\right)}(x) = \sum_{i=1}^{N_{\emptyset}} N_D^{(\theta_i I)} \left(r_i^{(\emptyset)}m_i^{(\emptyset)}x \right).
\end{align*}
\end{proposition}

\begin{proof}
Let $f$ be an eigenfunction of $\left(\mathcal{E},\ft_0,\mu^{(I)}\right)$ with eigenvalue $\lambda$. Then
\begin{align*}
\mathcal{E}(f,g) = \lambda \, \langle f,g \rangle_{L_2\left(\mu^{(I)}\right)},
\end{align*}
for all $g \in \ft_0$.  Therefore, we have with \cite[Proposition 3.2.1]{Arz14} and Lemma \ref{scaleprop E},
\begin{align*}
&\sum_{i=1}^{N_{\emptyset}} \frac{1}{r_i^{(\emptyset)}} \, \mathcal{E}\left(f \circ S_i^{(\emptyset)}, g \circ S_i^{(\emptyset)}\right) + \sum_{i=1}^{N_{\emptyset}-1} \int_{S_i^{(\emptyset)}(b)}^{S_{i+1}^{(\emptyset)}(a)} f'(t) \, g'(t) \, dt \\  &= \lambda \, \sum_{i=1}^{N_{\emptyset}} m_i^{(\emptyset)} \, \langle f,g \rangle_{L_2\left(\mu^{(\theta_i I)}\right)}.
\end{align*}
For $h \in \mathcal{F}_0$ we define
\begin{align*}
\tilde{h}_j(x) \coloneqq
\begin{cases}
h \circ S_j^{(\emptyset)^{-1}}(x), &\text{if } x \in S_j^{(\emptyset)}([a,b]) \\
0, &\text{otherwise}.
\end{cases}
\end{align*}
Because $h \in \mathcal{F}_0$, it follows $\tilde{h}_j \in \ft_0$ and $\tilde{h}_j \circ S_j^{(\emptyset)} = h$ for $j=1,...,N_{\emptyset}$ and $\tilde{h}_j \circ S_i^{(\emptyset)} = 0$, if $i \neq j$. Hence,
\begin{align*}
\frac{1}{r_j^{(\emptyset)}} \, \mathcal{E}\left(f \circ S_j^{(\emptyset)}, h\right) = \lambda \, m_j^{(\emptyset)} \, \left\langle f \circ S_j^{(\emptyset)},h \right\rangle_{L_2\left(\mu^{\left(\theta_j I\right)}\right)},
\end{align*}
for all $j = 1,...,N_{\emptyset}$. Therefore, $\lambda \, r^{(\emptyset)}_i m_i^{(\emptyset)}$ is an eigenvalue of $\left(\mathcal{E},\mathcal{F}_0, \mu^{(\theta_i I)}\right)$ with eigenfunction $f \circ S_i^{(\emptyset)}$, $i=1,...,N_{\emptyset}$.
\\\\
Now, let $r_i^{(\emptyset)} m_i^{\emptyset)} \lambda$ be an eigenvalue of $\left(\mathcal{E},\mathcal{F}_0,\mu^{(\theta_i I)}\right)$ for some $\lambda >0$ with corresponding eigenfunction $f_i$, $i = 1,...,N_{\emptyset}$. Therefore, we have
\begin{align*}
\mathcal{E} (f_i,g) = r_i^{(\emptyset)} m_i^{(\emptyset)} \lambda \, \langle f_i,g  \rangle_{L_2\left(\mu^{(\theta_i I)}\right)} 
\end{align*}
for all $g \in \mathcal{F}_0$. Let 
\begin{align*}
f(x) \coloneqq 
\begin{cases}
f_i \circ S_i^{(\emptyset)^{-1}}(x), &\text{if } x \in S_i^{(\emptyset)}([a,b]), \text{ for some } i \in \{1,...,N_{\emptyset}\} \\
0, &\text{otherwise}.
\end{cases}
\end{align*}
Since $f_i \in \mathcal{F}_0$, we have $f \in \ft_0$ and because of $ f \circ S_i^{(\emptyset)} = f_i$, $i=1,...,N_{\emptyset}$, we have
\begin{align*}
&\sum_{i=1}^{N_{\emptyset}} \frac{1}{r_i^{(\emptyset)}} \, \mathcal{E}\left(f \circ S_i^{(\emptyset)}, g \right) = \lambda \, \sum_{i=1}^{N_{\emptyset}} m_i^{(\emptyset)} \, \left\langle f \circ S_i^{(\emptyset)},g \right\rangle_{L_2\left(\mu^{(\theta_i I)}\right)},
\end{align*}
for all $g \in \mathcal{F}_0$. For $\tilde{g} \in \ft_0$, we have $\tilde{g} \circ S_i^{(\emptyset)} \in \mathcal{F}_0$, $i=1,...,N_{\emptyset}$. Analogously to the case with Neumann boundary conditions we get
with \cite[Proposition 3.2.1]{Arz14} and Lemma \ref{scaleprop E},
\begin{align*}
\mathcal{E} (f,\tilde{g}) = \lambda  \, \langle f, \tilde{g} \rangle_{L_2\left(\mu^{(I)}\right)}.
\end{align*}
Hence, $\lambda$ is an eigenvalue of $(\mathcal{E},\tilde{\mathcal{F}}_0, \mu^{(I)})$ with eigenfunction $f$ and, as before, we can now easily conclude the claim.
\end{proof}

Since $\left(\et, \ft, \mu^{(I)}\right)$ is an extension of $\left(\mathcal{E}, \mathcal{F}, \mu^{(I)}\right)$ and $\left(\mathcal{E}, \mathcal{F}_0, \mu^{(I)}\right)$ is an extension of $\left(\mathcal{E},\tilde{\mathcal{F}}_0, \mu^{(I)}\right)$, we get the following corollary.

\begin{corollary} \label{counting scale}
For all $x \ge 0$ holds
\begin{align*}
\sum_{i=1}^{N_{\emptyset}} N_D^{(\theta_i I)} \left(r_i^{(\emptyset)} m_i^{(\emptyset)}x\right) \le N_D^{(I)}(x) \le  N^{(I)}_N(x)  \le \sum_{i=1}^{N_{\emptyset}} N^{(\theta_i I)}_N\left(r_i^{(\emptyset)} m_i^{(\emptyset)}x\right).
\end{align*}
\end{corollary}

\subsection{Spectral Asymptotics.}

We define a probability space $(\Omega, \mathcal{B}, \mathbb{P})$ in which every atomic event indicates a random tree $I$. Let $\left(\tilde{\Omega}, \tilde{\mathcal{B}}, \tilde{\mathbb{P}}\right)$ be a probability space and $\tilde{U}_{\ii}$, $\ii \in \mathcal{G}$ be i.i.d. $J$-valued random variables.
The probability space we are interested in is defined as in \eqref{prob space}, meaning

\begin{align*}
(\Omega,\mathcal{B},\mathbb{P}) = \prod_{\ii \in \mathcal{G}} (\Omega_{\ii}, \mathcal{B}_{\ii}, \mathbb{P}_{\ii}),
\end{align*}
whereby $(\Omega_{\ii}, \mathcal{B}_{\ii}, \mathbb{P}_{\ii})$ are copies of $(\tilde{\Omega}, \tilde{\mathcal{B}}, \tilde{\mathbb{P}})$. We set $U_{\ii} = \tilde{U}_{\ii} \circ P_{\ii}$, $\ii \in \mathcal{G}$, where $P_{\ii}$ is the projection map onto the $\ii$-th component.
 $\omega \in \Omega$ indicates a random tree $I(\omega)$. If $(i_1,...,i_n) = \ii \in \mathcal{G}$ is such that $N_{U_{(i_1,...,i_{n-1})}(\omega)} < i_n$, then in the infinite tree $I(\omega)$, the $i_n$-th child of $(i_1,...,i_{n-1})$ is never born, i.e. $\ii \notin I(\omega)$. If we refer to the Dirichlet/Neumann eigenvalue counting function, we write $N_{D/N}^{(\omega)}$ instead of $N_{D/N}^{(I(\omega))}$. Also, we write $\theta_{\ii} \omega$, if we mean the sub tree $\theta_{\ii} I(\omega)$ of $I(\omega)$, rooted at $\ii \in I(\omega)$. 
is measurable.\\

We consider C-M-J branching processes with
\begin{align*}
\left(\xi_{\ii},L_{\ii}\right) = \left(  \, \sum_{i=1}^{N_{U_{\ii}}} \delta_{-\log \left(r^{(U_{\ii})}_i \, m^{(U_{\ii})}_i\right)},\, \max_{i \in \{1,...,N_{U_{\ii}}\}}  -\log\left(r^{(U_{\ii})}_i\, m^{(U_{\ii})}_i\right)\right),
\end{align*}
whereby $\delta_{y}(\cdot)$ denotes the dirac delta function $\delta(\cdot -y)$. Let $(z_t)_t$ denote the C-M-J branching process to the random characteristic 
\begin{align*}
\hat{\phi}_{\ii}(t) \coloneqq \xi_{\ii}(\infty) - \xi_{\ii}(t).
\end{align*}  
Then $z_t$ denotes the number of individuals born after time $t$ to mothers born before or at time $t$. We assume that Condition \ref{CB1} and Condition \ref{CB2} are satisfied and thus there exists a random variable $W$ such that 
\begin{align*}
\lim_{t \rightarrow \infty } e^{- \alpha t} z_t = W \, \nu_\alpha^{\hat{\phi}}(\infty) ~~~ a.s., ~~~~~ \nu_\alpha^{\hat{\phi}} ( \infty ) \coloneqq \frac{\int_0^\infty e^{- \alpha t} \, \mathbb{E}(\hat{\phi}(t)) \, dt}{\int_0^\infty t \, d \nu_\alpha(t)},
\end{align*}
or there exists a periodic function $G_\alpha^{\hat{\phi}}$ such that
\begin{align*}
z_t = W \, e^{\alpha t} \left(G_\alpha^{\hat{\phi}} + o(1)\right) ~~~ a.s.
\end{align*}
If we assume that $\mathbb{E} N_{U_{\emptyset}}^2 < \infty$, we have
\begin{align*}
\mathbb{E}(\xi_\alpha(\infty) \, \log^+ \xi_\alpha(\infty))  < \infty.
\end{align*}
Hence, by Theorem \ref{W bigger zero}, $W>0$ a.s. For the rest of this chapter we denote by $W$ this random variable.

With Corollary \ref{counting scale} we have for each $x \ge 0$
\begin{align*}
\sum_{i=1}^{N_{U_\emptyset}} N_D^{(\theta_i \omega)} \left(r_i^{(U_\emptyset)} m_i^{(U_\emptyset)}x\right) \le N_D^{(\omega)}(x) \le  N^{(\omega)}_N(x)  \le \sum_{i=1}^{N_{U_\emptyset}} N^{(\theta_i \omega)}_N\left(r_i^{(U_\emptyset)} m_i^{(U_\emptyset)}x\right).
\end{align*}
We consider the scaling property
\begin{align*}
&\sum_{i=1}^{\xi_\emptyset(\infty)} N_D^{(\theta_i \omega)} (r_i^{(U_\emptyset)} m_i^{(U_\emptyset)}x) \le N_D^{(\omega)}(x).
\end{align*}
We suppress the $\omega$ dependence and define
\begin{align*}
X_D(t) \coloneqq N_D\left(e^t\right).
\end{align*}
Therefore, we have
\begin{align*}
&\sum_{i=1}^{\xi_\emptyset(\infty)} X_D(t- \sigma_i) \le X_D(t) ~~~ a.s.
\end{align*}
As in \cite{Ham00} we extend the branching processes to $\{X^{\phi}(t): ~ -\infty < t < \infty\}$, where
\begin{align*}
X^{\phi}(t) \coloneqq \sum_{\textbf{i} \in I} \phi_{\theta_\textbf{i} \omega}(t-\sigma_{\ii})
\end{align*}
and $\phi_\omega$ is defined for all $t \in \mathbb{R}$ and $\omega \in \Omega$. For our purposes it is enough that $\phi_\omega$ is bounded and $\phi_\omega(t) =0$ for all $t < t_0(\omega)$, for some $t_0(\omega) \in \mathbb{R}$. As for the C-M-J branching processes, we have
\begin{align}
X^\phi(t) = \phi_\omega(t) + \sum_{i=1}^{\xi_\emptyset(t)} \leftidx{_{(i)}} X^\phi (t-\sigma_i), \label{Xphi property}
\end{align} 
where $ \left\{\leftidx{_{(i)}} X^\phi (t)\right\}_t$, $i=1,...,\xi_\emptyset(\infty)$ are branching processes with characteristic $\phi$ with the assumption that the population has initial ancestor $(i)$. Moreover, $ \leftidx{_{(i)}} X^\phi $ are i.i.d. copies of $X^\phi$, distributed like $X^\phi$ and independent of $U_\emptyset$ and $\xi_\emptyset$. We will suppress $(i)$, if it will not cause confusion. \\
We want to give a representation of $X_D$ such that $X_D= X^{\phi}$ for some bounded $\phi$. Let

\begin{align*}
\eta(t) \coloneqq X_D(t) - \sum_{i=1}^{\xi_\emptyset(\infty)} X_D(t-\log \tau_1(i))
\end{align*}

and 

\begin{align*}
\tilde{\eta} (t) \coloneqq \eta(t) \mathbbm{1}_{\{t \ge 0\}} + \sum_{i=1}^{\xi(\infty)} \leftidx{_{(i)}}X_D(t- \sigma_{i}) \mathbbm{1}_{\{0 \le t < \sigma_i \}}.
\end{align*} 
Then, we have $X_D = X^\eta$ and $X^{\tilde{\eta}}(t) = \mathbbm{1}_{[0,\infty)}(t) \, X^{\eta}(t)$ and thus both processes have the same asymptotic behavior as $t$ tends to infinity.

\begin{lemma} \label{convergence X}
Assume that 
\begin{align*}
\mathbb{E} N_{U_\emptyset}^2 < \infty.
\end{align*}

 Then, the Malthusian parameter of the process $\{X_D(t): ~ t \in \mathbb{R}\}$ is the unique solution $\gamma > 0$ of
\begin{align*}
\mathbb{E}\left(\sum_{i=1}^{N_{U_\emptyset}} \, \left(r^{(U_\emptyset)}_i m^{(U_\emptyset)}_i\right)^{\gamma}\right) = 1.
\end{align*}
If $\nu$ is non-lattice, then 
\begin{align*}
\lim_{t \rightarrow \infty} X_D(t) \, e^{- \gamma t} = v_{\gamma}^{\tilde{\eta}} (\infty) \, W ~~~ a.s.,
\end{align*}
where 
\begin{align*}
v_{\gamma}^{\tilde{\eta}}(\infty) \coloneqq \frac{\int_{-\infty}^\infty e^{-\gamma t} \, \mathbb{E}(\tilde{\eta} (t)) \, dt }{\int_0^\infty t \, e^{-\gamma t} \, d \nu (t)}.
\end{align*}
If $\nu$ is lattice with period $T$, then
\begin{align*}
X_D(t) = (G_\gamma^{\tilde{\eta}}(t) +  o(1))\, e^{\gamma t}\, W ~~~ a.s.,
\end{align*}
where $G$ is a periodic function with period $T$, given by
\begin{align*}
G_\gamma^{\tilde{\eta}}(t) = T \cdot  \frac{\sum_{j=-\infty}^\infty e^{- \gamma (t+jT)} \, \mathbb{E}(\tilde{\eta}(t+jT))}{\int_0^\infty t \, e^{-\gamma t} \, d \nu(t)}.
\end{align*}
\end{lemma}
\begin{proof}
Let 
\begin{align*}
f(s) \coloneqq \mathbb{E}\left(\sum_{i=1}^{N_{U_\emptyset}} \left(r^{(U_\emptyset)}_i m^{(U_\emptyset)}_i\right)^s\right).
\end{align*}
By dominated convergence, we see $f: [0,\infty) \longrightarrow \mathbb{R}$ is continuous and because $r^{(j)}_i m^{(j)}_i < 1$ for all $j \in J$, $i=1,...,N_j$, $f$ is strictly decreasing. Because $N_j \ge 2$, $j \in J$, we have
\begin{align*}
f(0) \ge  2
\end{align*}
and 
\begin{align*}
\lim_{s \rightarrow \infty} f(s) = 0.
\end{align*}
By continuity, there exists $\gamma >0 $ such that $f(\gamma) = 1$. Furthermore, $\gamma$ is the unique solution strictly bigger than zero and also the Malthusian Parameter of the general branching process under consideration.
The first moment of $\nu_\gamma$ is finite, since $\mathbb{E}N_{U_\emptyset} < \infty$. With $g(t) = t^{-2} \wedge 1$ Condition \ref{CB1} is satisfied since
\begin{align*}
\mathbb{E} \left( \sup_{t \ge 0} \frac{\xi_\gamma(\infty)- \xi_\gamma(t)}{g(t)}  \right) \le\mathbb{E} \left( \sup_{t \ge 0}\int_t^\infty \frac{1}{g(s)} \, d\xi_\gamma(s) \right) \le \sup_{t \ge 0} \left\{(1 \vee t^2) e^{- \gamma t} \right\} \mathbb{E} N_{U_{\emptyset}} < \infty.
\end{align*}
By \cite[Lemma 4.10]{Min18} there exists a deterministic constant $\tilde{c}>0$ such that
\begin{align}
X_D(t) \le \tilde{c} \, e^t. \label{X bound}
\end{align}
Further, from the Dirichlet-Neumann-bracketing follows that 
\begin{align*}
0 \le \eta (t) \le \sum_{i=1}^{N_{U_\emptyset}} \left( N_N^{(\theta_i I )}\left(r_i^{(\emptyset)} m_i^{(\emptyset)} e^t \right) -  N_D^{(\theta_i I )}\left(r_i^{(\emptyset)} m_i^{(\emptyset)} e^t\right) \right).
\end{align*}
With \cite[Proposition 5.5]{Fre04/05}
\begin{align*}
N_D(x) \le N_N(x) \le N_D(x) + 2,
\end{align*}
we thus receive
\begin{align}
\eta (t) \le 2 N_{U_\emptyset}. \label{eta bound}
\end{align}
Taking together \eqref{X bound} and \eqref{eta bound}, we receive
\begin{align*}
\tilde{\eta}(t) \le c N_{U_\emptyset},
\end{align*}
for some deterministic $c>0$. Therefore, Condition \ref{CB2} follows with $h(t) = e^{- \gamma t}$. The Lemma then follows from Theorem \ref{strong law}.

\end{proof}

\begin{theorem} \label{main theorem 1}
Assume that 
\begin{align*}
\mathbb{E} N_{U_\emptyset}^2 < \infty
\end{align*}
and let $\gamma > 0 $ be the unique solution of
\begin{align*}
\mathbb{E}\left(\sum_{i=1}^{N_{U_\emptyset}} \, \left(r^{(U_\emptyset)}_i m^{(U_\emptyset)}_i\right)^{\gamma}\right) = 1.
\end{align*}
Then, 
\begin{enumerate}
\item
If $\nu$ is non-lattice, then 
\begin{align*}
\lim_{x \rightarrow \infty} N_{D/N}(x) \, x^{- \gamma } = \nu_{\gamma}^{\tilde{\eta}} (\infty) \, W, ~~~ a.s.,
\end{align*}
where 
\begin{align*}
\nu_{\gamma}^{\tilde{\eta}}(\infty) \coloneqq \frac{\int_{-\infty}^\infty e^{-\gamma t} \, \mathbb{E}(\tilde{\eta} (t)) \, dt }{\int_0^\infty t \, e^{-\gamma t} \, d \nu (t)}.
\end{align*}
\item 
If the support of $\nu$ lies in a discrete subgroup of $\mathbb{R}$, then
\begin{align*}
N_{D/N}(x) = (G(\log(x)) +  o(1))\, x^{\gamma}\, W, ~~~ a.s.,
\end{align*}
where $G$ is a periodic function with period $T$, given by
\begin{align*}
G(t) = T \cdot  \frac{\sum_{j=-\infty}^\infty e^{- \gamma (t+jT)} \, \mathbb{E}(\tilde{\eta}(t+jT))}{\int_0^\infty t \, e^{-\gamma t} \, d \nu(t)}.
\end{align*}
\end{enumerate}

\end{theorem}

\begin{proof}
For the Dirichlet eigenvalue counting function, we simply rescale Lemma \ref{convergence X} by $x = \log(t)$ and hence the claim follows. The assertion for the Neumann eigenvalue counting function follows from the identity
\begin{align*}
N_D(x) \le N_N(x) \le N_D(x) + 2,
\end{align*}
see \cite[Proposition 5.5]{Fre04/05}.
\end{proof}

\subsection{Comparison between Random Recursive and Random Homogeneous Cantor Measures.} \label{compare}
We have seen the construction of the recursive Cantor sets and the corresponding recursive Cantor measures. Then, we randomized these sets and measures and showed that under some  regularity conditions the spectral exponent for the corresponding Krein-Feller-operator is almost surely given by the unique solution $\gamma_r > 0$ of 
\begin{align*}
\mathbb{E}\left(\sum_{i=1}^{N_{U_\emptyset}} \, \left(r^{(U_\emptyset)}_i m^{(U_\emptyset)}_i\right)^{\gamma_r}\right) = 1.
\end{align*}

In Theorem \ref{hom asym} we recalled the results of \cite{Arz14} about the spectral asymptotics for Krein-Feller-operators w.r.t. random homogeneous Cantor measures. The next proposition relates $\gamma_r$ to $\gamma_h$, where we assume that conditions \eqref{A1}-\eqref{A5} are satisfied.
\begin{proposition}
With the notation above and in Theorem \ref{hom asym}, we have $\gamma_h \le \gamma_r$ and equality if and only if there exists $\alpha > 0$ such that
\begin{align}
\sum_{i=1}^{N_j}\left(r_i^{(j)} m_i^{(j)} \right)^\alpha = 1, ~~~ \text{for all }j\in J. \label{cond eq h and r}
\end{align}
\end{proposition}
\begin{proof}
Let $x_j(\alpha) \coloneqq \sum_{i=1}^{N_j}\left(r_i^{(j)} m_i^{(j)} \right)^\alpha$, $j \in J$. With Jensen's inequality, we receive
\begin{align*}
\sum_{j \in J } p_j\, \log\left(x_j(\alpha) \right) \le \log\left(\sum_{j \in J } p_j \, x_j(\alpha) \right). 
\end{align*}
Since $\log$ is strictly increasing, we have equality if and only if $x_i(\alpha) = x_j(\alpha)=1$ for all $i,j \in J$. 
Now, let \eqref{cond eq h and r} not be satisfied. Then,
\begin{align*}
0 = \sum_{j \in J } p_j\, \log\left(x_j(\gamma_h) \right)< \log\left(\sum_{j \in J } p_j \, x_j(\gamma_h) \right). 
\end{align*}
As $\log\left(\sum_{j \in J } p_j \, x_j(\alpha) \right)$ decreases as $\alpha$ increases, the assertion follows.
\end{proof}

\begin{remark}
If $U_{\ii} = U_{\jj}$ for all $\ii, \jj \in I$ such that $|\ii| = |\jj|$, then the corresponding recursive Cantor measure is homogeneous. However, Theorem \ref{main theorem 1} makes no statement about the spectral asymptotics w.r.t. homogeneous Cantor measures, since the probability that $\mu^{(I)}$ is homogeneous is 0. 
\end{remark}

\begin{example} Let $J$ be countable and $p_j \coloneqq \mathbb{P}\left( U_\emptyset = j) \in (0,1\right)$, $j \in J$. Further, assume that $r_1^{(j)} = ... = r_{N_j}^{(j)}$, $m_1^{(j)} = ... = m_{N_j}^{(j)}$ for all $j \in J$. Therefore, $m_i^{(j)} = \frac{1}{N_j}$ $i=1,...,N_j$ for all $j \in J$. Let $r \coloneqq r_{U_\emptyset}$ and $N \coloneqq N_{U_\emptyset}$. If the conditions \eqref{A1}-\eqref{A5} are satisfied, then the spectral exponent for the Krein-Feller-operator w.r.t. the corresponding random homogeneous Cantor measure is given by 
\begin{align*}
\gamma_h \coloneqq \frac{\mathbb{E} \log N}{\mathbb{E}\log(N/r)}, 
\end{align*}
see \cite[Page 64]{Arz14}. The spectral exponent for the Krein-Feller-operator w.r.t. the corresponding random recursive Cantor measure is given by the unique solution $\gamma_r > 0$ of
\begin{align*}
\mathbb{E}\left( N \, (r/N)^{\gamma_r}\right) = 1.
\end{align*}
If not $ (r/N)^{\alpha} = 1/N$ for some $\alpha > 0$, for almost all $\omega \in \Omega$, we thus have
\begin{align*}
0 = \log\left( \sum_{j \in J} p_j \, N_j \, \left(r^{(j)}_1/N_j\right)^{\gamma_r}\right) < \sum_{j \in J} p_j  \log\left( N_j \, \left(r_1^{(j)}/N_j\right)^{\gamma_r} \right) = \mathbb{E}\log\left( N \, (r/N)^{\gamma_r} \right). 
\end{align*}
Therefore,
\begin{align*}
\gamma_h =  \frac{\mathbb{E} \log N}{\mathbb{E}\log(N/r)} < \gamma_r.
\end{align*}
\end{example}
Coming back to the $\frac{1}{3}$-$\frac{1}{5}$-recursive Cantor set from the introduction and let $p=\frac{3}{5}$, $m_1^{(1)}= m_2^{(1)} = \frac{1}{2}$, $m_1^{(2)} = m_2^{(2)} = m_3^{(2)} = \frac{1}{3}$. Then, the spectral exponent for the Krein-Feller-operator w.r.t. the corresponding random recursive Cantor measure is given as the unique solution $\gamma_r > 0$ of
\begin{align*}
\left(\frac{1}{6} \right)^{\gamma_r} + \left(\frac{1}{15}\right)^{\gamma_r} = \frac{5}{6}.
\end{align*}
Numerically, we get $\gamma_r \approx 0.396403$.

\end{document}